\documentclass[12pt]{article}
\usepackage{geometry}
\usepackage{amsfonts}
\usepackage[utf8]{inputenc}

\usepackage[english]{babel}
\usepackage[
backend=biber,
style=numeric,
sorting=nyt
]{biblatex}

\usepackage{amsmath, amsthm, amsfonts,amssymb,graphicx,csquotes,mathtools}
\usepackage{authblk}
\usepackage{float}
\usepackage{bigints}
\usepackage{tikz-cd}

\newtheorem{theorem}{Theorem}
\newtheorem{proposition}[theorem]{Proposition}
\newtheorem{lemma}[theorem]{Lemma}
\newtheorem{sublemma}[theorem]{Sublemma}

\theoremstyle{definition}

\newtheorem*{conjecture}{Conjecture}
\theoremstyle{remark}
\newtheorem{remark}{Remark}[section]

\title{Existence of absolutely continuous invariant measures for $C^1$ expanding circle maps.}
\author{Hamza Ounesli$^{1,2}$}

\date{%
    $^1$\small\textit{Scuola Internazionale Superiorie di Studi Avanzati (SISSA), Trieste, Italy.} \vspace{0.2cm}%
    $^2$\textit{Abdus Salam International Centre for Theoretical Physics (ICTP), Trieste, Italy.}\\(\textit{Email: hounesli@sissa.it/hounesli@ictp.it})\\[2ex]%
    \today
}

\begin{document}
\maketitle
\begin{abstract}
We prove that for any given modulus of continuity $\omega$ there exist (uncountably many) $C^1$ uniformly expanding maps of the circle whose derivatives have $\omega$ as an optimal modulus of continuity and which preserve an invariant probability measure equivalent to Lebesgue whose density is \( \omega\)-continuous, and also (uncountably many) $C^1$ uniformly expanding maps of the circle whose derivatives have $\omega$ as an optimal modulus of continuity which preserve Lebesgue measure. 
Moreover, we show that many of these maps, including those which preserve Lebesgue measure, have unbounded distortion. 
\end{abstract}
\section{Introduction and Statement of Results}

\subsection{Introduction and Background}
Let $E^1(\mathbb S^1)$ be the space of $C^1$ uniformly expanding maps on the circle. It is  essentially a Folklore Theorem dating back to the 1950s that if \( f\in E^1(\mathbb S^1)\) is \(C^{1+\alpha}\), i.e if the derivative is H\"older continuous, then \( f \) admits a unique ergodic invariant probability measure equivalent to Lebesgue. This result, together with the techniques involved in the proof, have led to a huge area of research and many generalizations to uniformly and non-uniformly expanding maps on manifolds of arbitrary dimension as well as to more general hyperbolic and non-uniformly hyperbolic systems. 

However, even in this simplest setting of uniformly expanding circle maps there are still open problems for maps with lower degrees of regularity.  Indeed G\'ora and  Schmitt  \cite{3}  constructed an example of a map \( f\in E^1(\mathbb S^1)\)  which does not admit any invariant probability measure absolutely continuous with respect to Lebesgue (\emph{acip}). Quas \cite{2} then showed that this is not an isolated example by proving that \emph{generically} in the $C^1$-topology, maps in $E^{1}(\mathbb S^1)$ have no \emph{acip} and, more recently  Avila and Bochi \cite{5} even showed that  generically in the $C^1$-topology, maps in $E^{1}(\mathbb S^1)$ do not even have an absolutely continuous invariant \( \sigma\)-\emph{finite} measure.

On the more ``positive'' side, it is possible to relax  the condition on the H\"older continuity of the derivative to some extent. Recall that the \emph{modulus of continuity} of a continuous map $\rho: X \to Y$  between two metric spaces is  a continuous  map $\omega: \mathbb{R}^{+} \to \mathbb{R}^{+}$ vanishing at $0$ and satisfying 
\begin{equation}\label{eq:modulus}
d_Y(\rho(x),\rho(y)) \leq \omega(d_X(x,y)) 
\end{equation}
for every \( x, y \in X \). We say that \( \omega\) is \emph{Dini-integrable} if 
\begin{equation*}\label{eq:thm2}
    \int_{0}^{1}\dfrac{\omega(t)}{t}dt<\infty.
\end{equation*}
Notice that saying that \( \rho\) is H\"older continuous is exactly equivalent to saying that \( \rho \) has a modulus of continuity of the form $\omega(t) = Ct^\alpha$ for some $\alpha \in (0,1)$ and that this implies in particular that \( \omega \) is Dini-integrable. 
Fan and Jiang \cite{1} showed that if the derivative of  \( f\in E^1(\mathbb S^1)\) has a modulus of continuity which is Dini-integrable then  \( f \) admits a unique ergodic invariant probability measure equivalent to Lebesgue, thus extending the Folklore Theorem to a lower degree of regularity of the map. 

All the counterexamples mentioned above must therefore have modulus of continuity for the derivative which is not Dini-integrable and a natural question is whether Dini-integrability defines a precise \emph{cut-off} between \( C^1\) uniformly expanding maps which admit and which do not admit an \emph{acip}. 
In this paper we explore this ``{underground}'' world of maps in \(  E^1(\mathbb S^1)\) with very low regularity, in particular whose derivative have modulus of continuity which is not Dini-integrable. We show that for \emph{any} given modulus of continuity \(\omega\) there are (uncountably many) maps in \(  E^1(\mathbb S^1)\) whose derivative has a modulus of continuity equivalent to \( \omega\) but nevertheless still admit an \emph{acip}. In particular there is no specific \emph{cut-off} based on the regularity of the derivative, which means that other characteristics of the map somehow come into play.

\subsection{Existence of acip}
To state our results  we define the \emph{canonical modulus of continuity} of $\rho$ by
\begin{equation*}
\omega_{\rho}(t):=\sup\{|\rho(x)-\rho(y)|: d(x,y)<t\}.
\end{equation*}
Notice that 
$\omega_{\rho}$ always exists if \( X \) is compact since every continuous function is uniformly continuous. It is also easy to check that $\omega_{\rho}$ is increasing, concave, and satisfies \eqref{eq:modulus}.
We define the space of all potential moduli of continuity by  
\[
K:=\{\omega\in C^0(\mathbb{R}^{+},\mathbb{R}^{+}): \text{continuous, increasing, concave, $\omega(0)=0$}\},
\] 
and define an equivalence relation on \( K \) by letting \( \omega \simeq \tilde\omega\) if the ratio  $\omega/\tilde \omega$ is uniformly bounded above and below. Then, following \cite{3},  we  say that \( \omega \in K \) is an \emph{optimal modulus of continuity} for \( \rho\) if it is \emph{equivalent to} \( \omega_{\rho}\).

\begin{remark}
Despite its name, the optimal modulus of continuity is not unique but rather defines a class of functions of which \( \omega_{\rho}\) is, in some sense, a canonical representative and such that all the moduli in this class have essentially the same behaviour near 0. For example if \( \rho \) is H\"older continuous and its canonical modulus is \( \omega_{\rho}(t) = Ct^{\alpha}\), for some \( C, \alpha>0 \), then any optimal modulus of continuity for \( \rho \) will have the form \( \omega(t) = \tilde C t^{\alpha}\) for some \( \tilde C > 0 \). 
\end{remark}

The equivalence relation on \( K \) defined above induces an equivalence relation on the space \( E^{1}(\mathbb S^1)\) by letting \( f\sim g\) whenever \( \omega_{f'} \simeq \omega_{g'}\), i.e. whenever the corresponding canonical moduli of the derivatives \( f', g'\) are equivalent. The  equivalence classes associated to this equivalence relation are  of the form 
\[
E^{1}_{\omega}(\mathbb S^1)\coloneqq \{ f\in E^{1}(\mathbb S^1): \omega_{f'}\simeq \omega\}
\] 
for   \( \omega\in K \). Indeed, notice that for \( \omega, \tilde\omega\in K\) we have that \( E^{1}_{\omega}(\mathbb S^1) = E^{1}_{\tilde\omega}(\mathbb S^1)\) if \( \omega \simeq \tilde\omega\) and \( E^{1}_{\omega}(\mathbb S^1) \cap E^{1}_{\tilde\omega}(\mathbb S^1) = \emptyset\) otherwise. Notice that \( E^{1}_{\omega}(\mathbb S^1) \) contains a large number of maps, as specifying only the modulus of continuity of \( f'\) leaves a lot of freedom in the definition of~\( f\). We are interested in the sets
\[
\Gamma_{\omega}^{1}(\mathbb S^1)\coloneqq \{ f\in E^{1}_{\omega}(\mathbb S^1): f \text{ admits an \emph{acip} equivalent to Lebesgue}\}.
\]
By \cite{1}, as mentioned above, if \( \omega\) is Dini-integrable, and therefore in particular if \( \omega\) is H\"older continuous, every \( f \in E^{1}_{\omega}(\mathbb S^1)\) admits an \emph{acip} equivalent to Lebesgue and therefore
 \( \Gamma_{\omega}^{1}(\mathbb S^1) =  E^{1}_{\omega}(\mathbb S^1)\). On the other hand,  if \( \omega \) is not Dini-integrable then by \cite{3}  there  exist examples of \( \omega\in K\) such that 
\( E^{1}_{\omega}(\mathbb S^1) \setminus \Gamma_{\omega}^{1}(\mathbb S^1) \), and \cite{2, 5} even seem to suggest that there may be examples of \( \omega\in K \) for which \( \Gamma_{\omega}^{1}(\mathbb S^1) = \emptyset\). Our main result shows that this is not the case and that, \emph{on the contrary, \( \Gamma_{\omega}^{1}(\mathbb S^1) \neq \emptyset \) for every \( \omega\in K\)}.  Moreover, our arguments are quite constructive and  yield additional information about the possible regularities of the densities of the \emph{acip}, and in particular show that their regularity may be as low as that of \( f \) itself, i.e. have \( \omega \) as an optimal modulus of continuity, or very smooth, including cases in which \emph{Lebesgue measure itself is invariant}. 
For every   $\omega\in K$ and  \( f\in \Gamma_{\omega}^{1}(\mathbb S^1)  \), we let \( \mu_{f}\) denote the \emph{acip} equivalent to Lebesgue, let \( \rho_{f}= d\mu_{f}/dm\) denote its (continuous) density with respect to Lebesgue, and let \( \omega_{\rho_{f}}\) denote the canonical modulus of continuity of \( \rho_{f}\).

\begin{theorem}\label{thm:main}
For every   $\omega\in K$  there exists an uncountable set in 
\( \Gamma_{\omega}^{1}(\mathbb S^1)  \) for which \( \omega_{\rho_{f}}\simeq \omega\) and which can be given in a relatively explicit way, see \eqref{eq:map}.
\end{theorem}

The main point of Theorem~\ref{thm:main} is the fact that  \( \Gamma_{\omega}^{1}(\mathbb S^1) \neq\emptyset \), which means that even  maps in \( E^{1}(\mathbb S^1) \) with \emph{arbitrarily low regularity} can  admit an  \emph{acip} and also implies that distinct maps with equivalent moduli of continuity can have quite different ergodic properties. Indeed it implies that \( \Gamma_{\omega}^{1}(\mathbb S^1) \neq\emptyset \) in particular  for the specific modulus of continuity \( \omega\) of the counterexample constructed in \cite{3} which however does not admit an \emph{acip}. The additional statements about the densities of the \emph{acip} highlight the fact that \( \Gamma_{\omega}^{1}(\mathbb S^1)  \) is in fact quite a large set and that there is a remarkable flexibility in the construction of examples with various kinds of densities. The fact that \( f\in \Gamma_{\omega}^{1}(\mathbb S^1)  \) can preserve a density whose modulus of continuity is equivalent to the modulus of~\( f' \) seems quite natural but turns out to be somewhat coincidental as we show that there exists also maps  \( f\in \Gamma_{\omega}^{1}(\mathbb S^1)  \)  which preserve densities which are much more regular than that of \( f'\), even Lebesgue measure itself.

\begin{theorem}\label{thm:lebesgue}
Let $a\in(0,1)$ and let $f_1:[0,a]\to[0,1]$ be an expanding $C^1$-diffeomorphism. Then there exists a unique extension of $f_1$ to a Lebesgue-preserving full branch expanding transformation of the unit interval. This extension represents a $C^1$ map on the circle if and only if the following holds:
    \begin{equation}\label{The condition to extend a full branch to an expanding map}
        f_{1}'(0)=\dfrac{f_{1}'(a)}{f_1'(a)-1}
    \end{equation}
In particular, for every   $\omega\in K$ there exists an uncountable set in 
\( \Gamma_{\omega}^{1}(\mathbb S^1)  \) for which \( \mu_{f}\) is Lebesgue.
\end{theorem}

For for future reference, for every $\omega\in K$ we let  
 \[
 \Gamma_{\omega,\lambda}(\mathbb{S}^{1})\coloneqq\{f\in \Gamma^{1}_{\omega}(\mathbb{S}^1) :\text{Lebesgue measure is invariant}\} 
 \]

\subsection{Bounded and unbounded distortion}
One of the main techniques for proving the existence of an \emph{acip} is through a bounded distortion property. For  $f\in E^{1}(\mathbb S^1)$ we let  \( \{ \omega^{(n)}_{i}\}\)  denote the injectivity domains of $f^{n}$ and say that \( f \) has \emph{bounded distortion} if 
\begin{equation*}
    \mathcal{D} \coloneqq \sup_{n\geq 1}  \sup_{\omega_{i}^{(n)}}\sup_{x,y\in \omega_{i}^{(n)}}
    \log\dfrac{(f^n)'(x)}{(f^n)'(y)} < \infty.
\end{equation*}
It is possible to show that if \( \omega \) is Dini-integrable then every  $f\in E^{1}_{\omega}(\mathbb S^1)$ has  bounded distortion and therefore, since by classical arguments  bounded distortion implies the existence of an \emph{acip} equivalent to Lebesgue, this implies that  \( \Gamma_{\omega}^{1}(\mathbb S^1) =  E^{1}_{\omega}(\mathbb S^1)\), as mentioned above. If \( \omega \) is not Dini-integrable then bounded distortion cannot be guaranteed and indeed  our construction of the \emph{acip} for maps for maps $f\in E^{1}_{\omega}(\mathbb S^1)$ in this setting  does not explicitly use any distortion estimates. An interesting question therefore is whether Dini-integrability is a necessary as well as a sufficient condition for uniformly bounded distortion and, if not, whether there is actually is any underlying bounded distortion property which is   implicitly responsible for the existence of an \emph{acip} in the cases given by Theorem \ref{thm:main}. 

\begin{conjecture}
$\forall\ \omega \in K$ non Dini-integrable, unbounded distortion is $C^1$-generic in \( \Gamma_{\omega}^{1}(\mathbb S^1)\)
\end{conjecture}

While we cannot give a full answer to the conjecture we can show that many maps have an \emph{acip} despite not having bounded distortion. 
For \( \omega \in K \) we  consider a subset of the family \( \Gamma_{\omega,\lambda}(\mathbb{S}^{1})\) defined above for which the derivative has an explicit form near 0. 
 \[
 \mathcal{F}_{\omega}\coloneqq \{f\in \Gamma_{\omega,\lambda}(\mathbb{S}^{1}) :\text{$f_1'(x)=2+2\omega(x)$ on a small enough interval $[0,t_{\omega}]$}\} 
 \]
 It is clear by the statement in Theorem \ref{thm:lebesgue} that \(  \mathcal{F}_{\omega} \) is an uncountable set.  
 
\begin{theorem}\label{thm:dinidist1}
  Every map in $\mathcal{F}_{\omega}$ has bounded distortion if and only if the optimal modulus of continuity $\omega_{f'}$ of $f'$ is Dini-integrable.
\end{theorem}

Finally, also in the direction  of the Conjecture above, we show that unbounded distortion is generic in a somewhat different sense. More precisely, we define on $E^{1}(S^1)$  the $C^{1+mod}$-topology induced by the  metric
\begin{equation*}
d_{1+mod}(f,g)=d_{1}(f,g)+d_{0}(\omega_{f'},\omega_{g'}),
\end{equation*}
where $d_{1}$ is the $C^{1}$ distance, and $d_{0}$ is the $C^{0}$-distance. In the distance $d_{1+mod}$, maps are close if they are $C^{1}$-close, and their moduli of continuity $\omega_{f'}$ and $\omega_{g'}$ of their derivatives are close in the $C^{0}$-topology. Notice that this is a  natural metric on the space of \( C^{1}\) maps and stronger than the usual \( C^{1}\) metric. 

\begin{theorem}\label{thm:generic}
There exists a subset $\Gamma \subset E^{1}(S^{1})$ which contains exactly one element from each equivalence class  $ E^{1}_{\omega}(S^{1})$ for \( \omega \in K \), such that  \(C^{1+mod}\) generic maps  \( f\in \Gamma\) have unbounded distortion.
\end{theorem}

\begin{remark}
This theorem implies, in particular, that most maps in $\Gamma$ have unbounded distortion and still preserve a continuous probability measure equivalent to Lebesgue. Such examples are rare to find in the literature. The only example we know of is the Quas example in \cite{8} where he constructed an expanding map of the circle preserving Lebesgue but not ergodic and hence has unbounded distortion.
\end{remark}
\vspace{0.3cm}

\begin{center}
\textbf{Acknowledgments.}
\end{center}

\noindent I would first like to thank my supervisor Stefano Luzzatto for his guidance and encouragement during all stages of the writing of this note. I would also like to thank Anthony Quas, Ali Tahzibi, Houssam Boukhecham, and Douglas Coates for reading early versions of this note and for their useful comments. Finally, I am thankful to Benoit Kloeckner and Houssam Boukhecham for having hosted me at UPEC, Paris, where part of this work was done.

\section{Proof of part 1 of Theorem 1}

Let $\omega\in K$, we will construct uncountably many maps in $\Gamma^1_{\omega}(\mathbb{S}^1)$ for which the density they preserve has $\omega$ as an optimal modulus of continuity. We will construct these as maps $f:\mathbb{S}^{1} \rightarrow \mathbb{S}^{1}$ of degree 2, orientation-preserving which we represent as full branch map of the unit interval $[0,1]$ with two $C^1$ branches $f_1$ and $f_2$ defined respectively on $[0,\frac{1}{2}]$ and $[\frac{1}{2},1]$ satisfying $f'_1(0)=f'_2(0)$ and $f'_{l,1}(\frac{1}{2})=f'_{r,2}(\frac{1}{2})$ where $l$ and $r$ denote the left and right derivatives at $x=\frac{1}{2}$.

\begin{figure}[H]
\centering
\includegraphics[scale=0.5]{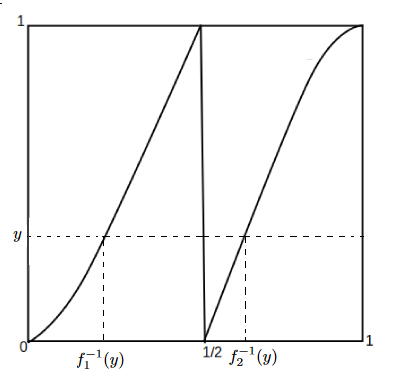}
\caption{A representation of a circle map of degree 2.}
\end{figure}

\noindent We will first give an overview of the proof and reduce it to a number of technical propositions which we will prove in the subsequent sections.
\subsection{Overview of the proof}

\noindent Our idea is to fix a continuous density satisfying certain conditions and prove that under those conditions we can construct a uniformly expanding map of the circle preserving the measure defined by that density and for which the regularity of the derivative is the same as that of the density.

\begin{lemma}\label{lem:rho1}
    For every $\omega\in K$ there exists $\rho:[0,1]\to\mathbb{R}$ continuous, having $\omega$ as an optimal modulus of continuity, strictly greater than $1\slash 2$, satisfying:
    
     \begin{equation*}
         \int_{0}^{\frac{1}{2}}\rho(t)dt=\int_{\frac{1}{2}}^{1}\rho(t)dt=\frac{1}{2},\tag{P1}
    \end{equation*}
     \begin{equation*}
        \max\limits_{[0,1]}\rho-\min\limits_{[0,1]}\rho<\frac{1}{2},\tag{P2}
    \end{equation*}
     and
    \begin{equation*}
        \rho(0)=\rho(1)=1\tag{P3}.
    \end{equation*}
 
\end{lemma}\label{lem:rho}
  Now, assuming the conditions of the previous lemma, for $x\in[0,1]$ let:
  \begin{equation*}
      g(x)=\int_{0}^{x}\rho(t)dt
  \end{equation*}   
and define $f_{\rho}:[0,1]\to[0,1]$ by
\begin{equation}   \label{eq:map}
 f_{\rho}(x) =
\begin{cases}
		2g  & \mbox{if } x \in [0,\frac{1}{2}] \\
		(g-\frac{1}{2}I)^{-1}\circ (g-\frac{1}{2}) & \mbox{if } x\in[\frac{1}{2},1]
	\end{cases}
\end{equation}

We will show that the map $f_{\rho}$ is a well defined $C^1$ expanding circle map which preserves the density $\rho$ and whose derivative $f'_{\rho}$ has $\omega$ as an optimal modulus of continuity, thus proving part 1 of Theorem 1.

 We split the proof into the following propositions. First of all let $\rho:[0,1]\to\mathbb{R}$ be a continuous map such that $\rho>1/2$ and consider the following system of ordinary differential equations:

\begin{equation*}
    \begin{cases}
    f'_{1}=2\rho\ & \text{on}\ [0,\frac{1}{2}]\ \text{with}\ f_1(0)=0, \\  f'_{2}=\dfrac{2\rho}{2\rho\circ f_{2}-1}\ &\text{on}\ [\frac{1}{2},1] \text{with}\ f_2(\frac{1}{2})=0.
    \end{cases}\tag{S}
\end{equation*}

\vspace{0.1cm}

\begin{proposition}\label{propo:p1}
    If $\rho$ satisfies $(P1)$ then the system $(S)$ has a solution that defines a full branch map $f$ of the unit interval $[0,1]$ which preserves the measure $\mu$ defined by the density $\rho$.
\end{proposition}

\begin{proposition}\label{propo:p13}
     If $\rho$ satisfies $(P1)-(P3)$ then the map previously constructed coincides with $f_{\rho}$ and represents a $C^1$ uniformly expanding map of the circle. 
\end{proposition}

\begin{proposition}\label{propo:mod}
    Let $\omega\in K$ and $\rho$ be the function given by Lemma \ref{lem:rho1}. Then $f'_{\rho}$ has $\omega$ as an optimal modulus of continuity, and in particular $f_{\rho}\in \Gamma^{1}_{\omega}(\mathbb{S}^1)$.
\end{proposition}

\subsection{Proof of proposition \ref{propo:p1}}

We will split the proof of the proposition to 3 lemmas.

\begin{lemma}
    If $\rho>1/2$ and satisfies $(P1)$  Then $f_{\rho}$ is a well defined full branch map of the interval.
\end{lemma}

\begin{proof}
    First, notice that by definition $g(0)=0$ and by $(P1)$ we have $g(\frac{1}{2})=1$, so we obtain that $f_{\rho}$ maps diffeomorphically $[0,\frac{1}{2}]$ to $[0,1]$.
    Now notice that $g-\frac{1}{2}$ maps $[\frac{1}{2},1]$ to $[0,\frac{1}{2}]$ and since $g'=\rho>\frac{1}{2}$ then $(g-\frac{1}{2}I)$ is a diffeomorphism which maps $[0,1]$ to $[0,\frac{1}{2}]$ and hence $f_{\rho}$ maps diffeomorphically $[\frac{1}{2},1]$ to $[0,1]$. We conclude that our map is well defined and full branch on the interval $[0,1]$.
\end{proof}

\begin{lemma}
    Under the previous conditions, $f_{\rho}$ is a solution to the system $(S)$
\end{lemma}

\begin{proof}
    Let us recall that $g$ is the map defined on $[0,1]$ by:

\begin{equation*}
    g(x)=\int_{0}^{x}\rho(t)dt.
\end{equation*}

\noindent Clearly, $f'_{\rho,1}(x)=2\rho(x)$, now we have:

to show the other equality, notice that
\begin{equation*}
    f'_{2}=\dfrac{2\rho}{2\rho\circ f_{2}-1}
\end{equation*}
\noindent is equivalent to:

\begin{equation*}
    2f'_{2}\rho\circ f_{2}-f'_{2}=2\rho \iff 2(g\circ f_{2}-\dfrac{1}{2}f_{2})'=2g'
\end{equation*}after integrating over $[\dfrac{1}{2},x]$ we obtain:

\begin{equation*}
    (g-\dfrac{1}{2}I)\circ f_{2}(x)=g(x)-\dfrac{1}{2}.
\end{equation*}

\noindent where $I$ denotes the identity map, notice that $(g-\dfrac{1}{2}I)'>0$ and hence $g-\dfrac{1}{2}I$ is invertible, we obtain finally:

\begin{equation*}
    f_{2}=f_{\rho,2}=(g-\dfrac{1}{2}I)^{-1}\circ(g-\dfrac{1}{2}).
\end{equation*}
and so we conclude that the system (S) admits $f_{\rho}$ as a solution.
\end{proof}

\begin{lemma}
  If a solution of (S) is full branch then it preserves the measure $\mu$ defined by the density $\rho$.  
\end{lemma}

\begin{proof}
     We start by recalling the following sublemma:

\begin{sublemma}
    If $f_{\star}\mu([0,y])=\mu([0,y])$ for every $y\in[0,1]$ then $\mu$ is $f$-ivariant.
\end{sublemma}

\begin{proof}
    The $\sigma$-algebra of Lebesgue measurable sets is generated by intervals of the form $[0,y]$ and all subsets of Borel sets of zero measure, since $f$ is $C^1$ then it already preserves sets of measure zero, and so if the assumption of the lemma is satisfied then $\mu$ if $f$-invariant.
\end{proof}
  \noindent Now let $y\in[0,1]$ and consider $f$ to be a full branch map solution to (S) on the unit interval $[0,1]$, since the derivative is everywhere positive, the branches are injective and so every pre-image contains exactly two points, therefore, we have that $f^{-1}(\lbrace y\rbrace)=\lbrace f^{-1}_{1}(y), f^{-1}_{2}(y) \rbrace$ such that $f^{-1}_{1}(y)\in [0,\frac{1}{2}]$ and $f^{-1}_{2}(y)\in[\frac{1}{2},1]$, of course, we are assuming for simplicity here that the middle point of the interval is the end point of the first branch, we obtain:

\begin{equation}
    f^{\star}\mu([0,y])=\mu(f^{-1}([0,y]))=\mu([0,f^{-1}_{1}(y)])+\mu([\frac{1}{2},f^{-1}_{2}(y)]).
\end{equation}


\begin{equation*}
 \phi_{1}(x)=\frac{1}{2}x\ \text{and}\ \phi_{2}(x)=\int_{0}^{x}(\rho(t)-\frac{1}{2})dt. 
\end{equation*}

\noindent Clearly $\mu([0,y])=\phi_{1}(y)+\phi_{2}(y)$ and that $\phi_{1}$ maps $[0,1]$ to $[0,\frac{1}{2}]$ and $\phi_{2}$ maps $[0,1]$ to $[0,\frac{1}{2}]$ because $\phi_{2}(2)$ is increasing since $\rho>\frac{1}{2}$, by definition also $\phi_{2}(0)=0$ and by $(P1)$

\begin{equation*}
    \phi_{2}(1)=\int_{0}^{1}\rho(t)-\frac{1}{2}dt=\int_{0}^{1}\rho(t)dt-\frac{1}{2}=1-\frac{1}{2}=\frac{1}{2}.
\end{equation*}

\noindent Now we want to solve the following equations:

 \begin{equation}\label{eq:maineq}
     \mu([0,f^{-1}_{1}(y)])=\phi_{1}(y)\ \text{and}\ \mu([\frac{1}{2},f^{-1}_{2}(y)])=\phi_{2}(y)
 \end{equation}

 \noindent Which is equivalent to:

 \begin{equation*}
     \int_{0}^{f_{1}^{-1}(y)}\rho(t)dt=\phi_{1}(y)\ \text{and}\ \int_{\frac{1}{2}}^{f_{2}^{-1}(y)}\rho(t)dt=\phi_{2}(y).
 \end{equation*}

 \noindent by differentiating both sides of the previous equations and using the formula:

 \begin{equation}
     \dfrac{d}{dy}\int_{\alpha}^{u(y)}v(x)dx=u'(y).v(u(y)).
 \end{equation}

\noindent we obtain the following two equations:
\begin{equation*}
    (f^{-1}_{1})'\rho\circ f^{-1}_{1}=\frac{1}{2}\ \text{and}\ (f^{-1}_{2})'\rho\circ f^{-1}_{2}=\rho-\frac{1}{2}.
\end{equation*}

\noindent using that $(f^{-1})'=1\slash f'\circ f^{-1}$ we obtain:

\begin{equation*}
    \dfrac{\rho}{f_1'}\circ f^{-1}_{1}=\dfrac{1}{2}\ \text{and} \  \dfrac{\rho}{f_2'}\circ f^{-1}_{2}=\rho-\dfrac{1}{2}.
\end{equation*}

\noindent compositing the left hand side of the previous equation by $f_1$ and the right side by $f_{2}$ we obtain exactly the system (S) but without particular initial conditions due to the differentiation step prior to obtaining (4) and so we are not sure the solution corresponds exactly to equation \eqref{eq:maineq}, we will show that the initial conditions of system (S) are sufficient to obtain \eqref{eq:maineq} and hence complete the proof. Notice that

\begin{equation*}
    \dfrac{d}{dy} \int_{0}^{f_{1}^{-1}(y)}\rho(t)dt=\phi'_{1}(y)\ \text{and}\ \dfrac{d}{dy}\int_{\frac{1}{2}}^{f_{2}^{-1}(y)}\rho(t)dt=\phi'_{2}(y).
\end{equation*}

\noindent Let us integrate the left and right hand side on $[0,z]$,  precisely:

\begin{equation*}
      \int_{0}^{z}(\dfrac{d}{dy} \int_{0}^{f_{1}^{-1}(y)}\rho(t)dt)dy=\int_{0}^{z}\phi'_{1}(y)dy,
\end{equation*}
\noindent and
\begin{equation*}
    \int_{0}^{z}(\dfrac{d}{dy}\int_{\frac{1}{2}}^{f_{2}^{-1}(y)}\rho(t)dt)dy=\int_{0}^{z}\phi'_{2}(y)dy.
\end{equation*}

\noindent this is equivalent to:

\begin{equation*}
    \int_{0}^{f_{1}^{-1}(z)}\rho(t)dt-\int_{0}^{f_{1}^{-1}(0)}\rho(t)dt=\phi_{1}(z)-\phi_{1}(0)=\phi_{1}(z).
\end{equation*}
\noindent and
\begin{equation*}
    \int_{\frac{1}{2}}^{f_{2}^{-1}(z)}\rho(t)dt-\int_{\frac{1}{2}}^{f_{2}^{-1}(0)}\rho(t)dt=\phi_{2}(z)-\phi_{2}(0)=\phi_{2}(z).
\end{equation*}

\noindent since the initial conditions are $f_1^{-1}(0)=0$ and $f_2^{-1}(0)=\frac{1}{2}$ we obtain finally the following:

\begin{equation*}
    \mu([0,f_1^{-1}(z)])=\phi_{1}(z)
\end{equation*}
and
\begin{equation*}
     \mu([\frac{1}{2},f_2^{-1}(z)])=\phi_{2}(z)
\end{equation*}.
This shows that (3) is satisfied and hence finishes the proof.
    \end{proof}
By the previous lemmas we proved, the map $f_{\rho}$ is a full branch map which preserves the measure $\mu$ defined by $\rho$, hence finishing the roof of proposition 5.

\subsection{Proof of proposition \ref{propo:p13}}

\begin{proof}
    
    \noindent Since $\rho>\frac{1}{2}$ we have that $f'(x)=2\rho(x)>1$ for every $x\in[0,\frac{1}{2}]$, now by $(P2)$ we also have that:

    \begin{equation*}
        \dfrac{2\min\limits_{[0,1]}\rho}{2\max\limits_{[0,1]}\rho-1}>\dfrac{2\min\limits_{[0,1]}\rho}{2(\min\limits_{[0,1]}\rho+\frac{1}{2})-1}=1
    \end{equation*}

    \noindent and hence we obtain that for every $x\in [\frac{1}{2},1]$ that $f_2'(x)>1$ and so $f$ is uniformly expanding, it remains to prove that $f$ represents a $C^1$ map of the circle as explained at the beginning of the proof of the theorem, namely, we have that  $f'_{l,1}(\frac{1}{2})=f'_{r,2}(\frac{1}{2})$ because by $(P3)$:

\begin{equation*}
    f'_{r,2}(\frac{1}{2})=\dfrac{2\rho(\frac{1}{2})}{2\rho(0)-1}=2\rho(\frac{1}{2})= f'_{l,1}(\frac{1}{2}).
\end{equation*}

\noindent Finally, we have $f'(0)=2\rho(0)=2$ and $f'(1)=\dfrac{2\rho(1)}{2\rho(1)-1}=2$. This shows indeed that the solution $f$ defines a $C^1$ uniformly expanding map on the circle and preserves $\mu$.
\end{proof}

\subsection{Proof of proposition \ref{propo:mod}}

  \begin{proof}

\noindent The regularity of the derivative on the first branch is clearly $\omega$-continuous. For the second branch we have that by definition:

\begin{equation*}
     \omega_{f'_2}(t)=\sup\limits_{0<\vert x-y\vert<t}\vert f'_2(x)-f'_2(y)\vert
\end{equation*}
\begin{equation*}
    =\sup\limits_{0<\vert x-y\vert<t}\vert\dfrac{2\rho(x)}{2\rho\circ f_{2}(x)-1}-\dfrac{2\rho(y)}{2\rho\circ f_{2}(y)-1}\vert
\end{equation*}
\begin{equation*}
    =\sup\limits_{0<\vert x-y\vert<t}|\rho(x)-\rho(y)||\dfrac{\dfrac{2\rho(x)}{2\rho\circ f_{2}(x)-1}-\dfrac{2\rho(y)}{2\rho\circ f_{2}(y)-1}}{\vert \rho(x)-\rho(y)\vert}|
\end{equation*} 

\noindent The term $|\dfrac{\dfrac{2\rho(x)}{2\rho\circ f_{2}(x)-1}-\dfrac{2\rho(y)}{2\rho\circ f_{2}(y)-1}}{\vert \rho(x)-\rho(y)\vert}|$ is uniformly bounded above and below  away from 0 in $[0,1]\times[0,1]$ because:

\begin{equation*}
   \vert\dfrac{\dfrac{2\rho(x)}{2\rho\circ f_{2}(x)-1}-\dfrac{2\rho(y)}{2\rho\circ f_{2}(y)-1}}{\vert \rho(x)-\rho(y)\vert}\vert
\end{equation*}
\begin{equation*}
   = \dfrac{1}{(2\rho\circ f_{2}(x)-1)(2\rho\circ f_{2}(y)-1)}\vert 2-\dfrac{4(\rho(x)\rho\circ f_{2}(y)-\rho(y)\rho\circ f_{2}(x)) }{\rho(y)-\rho(x)}\vert. 
\end{equation*}

\noindent The term $\dfrac{1}{(2\rho\circ f_{2}(x)-1)(2\rho\circ f_{2}(y)-1)}$ is bounded above and below away from zero because  of property $(P2)$ and that $\rho>1/2$, again for the same reasons we obtain that

\begin{equation*}
\vert2- \dfrac{4(\rho(x)\rho\circ f_{2}(y)-\rho(y)\rho\circ f_{2}(x)) }{\rho(y)-\rho(x)}\vert 
\end{equation*}

\noindent is also bounded above and below away from zero, this shows that for some $\alpha,\beta>0$ we get $\alpha\omega\leq \omega_{f'_2}(t)\leq \beta\omega$ and thus $f\in E_{\omega}^{1}(S^1)$.
\vspace{0.1cm}
  \end{proof}  

\subsection{Proof of Lemma \ref{lem:rho1}}

\noindent To finish the proof of the theorem, it clearly only remains to prove Lemma \ref{lem:rho1}.

\begin{proof}
    The existence of $\rho$ is guaranteed because on a small enough interval $[0,t_{\omega}]$ on which $\omega(t)<<\frac{3}{2}$ we can chose $\rho(t)$ to be equal to $1+\omega(t)$ for every $t\in [0,t_{\rho}]$, since the translation of an element in $K$  admits itself as an optimal modulus of continuity, it is enough to extend it to the rest of the interval in a $C^1$ way while satisfying the other properties which do not depend on how we define $\rho$ on a small enough neighborhood of $0$ as far as on that neighborhood $(P2)$ and $(P3)$ are satisfied there. We will in fact consider in the rest of the paper $\rho$ being defined on some neighborhood $[0,t_{\omega}]$ as $1+\omega(t)$.
\end{proof}

\section{Proof of  Theorem \ref{thm:lebesgue}}

We split the proof into two parts. We first prove the first statement concerning the extension of an arbitrary diffeomorphism to a Lebesgue measure preserving circle map, and then show that this map has optimal modulus of continuity \(\omega\)

\subsection{Extensions which preserve Lebesgue measure}
 The philosophy of the proof will be similar to the proof of   Theorem \ref{thm:main} but we will introduce a new ordinary differential equation that arises naturally from the transfer operator. 
 Let $f\in E^1(S^1)$ and,  for all $h\in L^1_{\lambda}(S^1)$ and  $\mu_{h} \coloneqq h\cdot \lambda$, we define the transfer operator associated to $f$ and acting on $L^1_{\lambda}(S^1)$ as
\begin{equation}
    Ph=\dfrac{d\big(f_{*}\mu_{h}\big)}{d\lambda}. 
\end{equation}
This operator can be interpreted as the density of the push-forward of measures in respect to Lebesgue. It is well known that the fixed points of $P$ corresponds to the densities of $f$-invariant measures and that the transfer operator for maps of degree 2 has an explicit formula given by 
\begin{equation}
    Ph(x)=\sum\limits_{y\in f^{-1}(x)}\dfrac{h(y)}{f'(y)}.
\end{equation}
We now suppose that $f_2:[a,1]\to \mathbb{R}$ is a diffeomorphism into its image and  consider the differential equation 
\begin{equation}\label{transfer operator equation...}
    \frac{1}{f_1'\big( f_1^{-1}(x)\big)}+\frac{1}{f_2'\big(f_2^{-1}(x)\big)}=1,\quad ~x\in[0,1].
\end{equation}
 The equation \eqref{transfer operator equation...} is equivalent to 
\begin{equation}\label{simplified transfer operator equation...}
f_2'(x)=\frac{f_1'\Big(f_1^{-1}\big(f_2(x)\big)\Big)}{f_1'\Big(f_1^{-1}\big(f_2(x)\big)\Big)-1}, \quad ~x\in[a,1],
\end{equation}
and since $f_1$ is $C^1,$ by Peano's existence theorem the Cauchy problem with the initial condition $f_2(a)=0$ admits a maximal solution $f_2$ defined on the interval $[a,1]$. Let's show that $f_2$ maps diffeomorphically onto $[0,1]$. Notice that $f_2'(x)>1$ for all $x\in [a,1]$, therefore it only remains to show that $f_2(1)=1.$ Assume that $f_2(1)<1,$ and consider $I=[f_{2}(1),1]$. By construction the map $f:[0,1]\to[0,1]$ defined by $f_1$ and $f_2$ preserves Lebesgue measure since $(3)$ corresponds to $(5)$ by taking $h$ to be the constant function 1, so $\lambda(I)=\lambda(f^{-1}(I))=\lambda(f^{-1}_1(I))$ because $f^{-1}_2(I)=\emptyset$, this is a contradiction because $f^{-1}_1$ is a contraction. We conclude that $f$ is indeed a uniformly expanding full branch map of the interval.

\noindent Uniqueness cannot be deduced directly from the equation (\ref{simplified transfer operator equation...}), because Peano's existence theorem provides only existence, we will deduce it using the fact that the solution preserves $\lambda$. Let $f,g:[0,1]\to[0,1]$ be two full branch interval maps which preserve Lebesgue measure, assume they have the same first branches (i,e $f_{1}=g_{1}$) on an interval $[0,a]$, we obtain that for every $y\in[0,y]$:

\begin{equation*}
    \lambda([0,y])=\lambda(f^{-1}([0,y]))=\lambda(g^{-1}([0,y])).
\end{equation*}

\noindent which implies by assumption:

\begin{equation*}
    \lambda([a,f^{-1}_{2}(y)])=\lambda([a,g^{-1}_2(y)])\Leftrightarrow f^{-1}_{2}(y)=g^{-1}_{2}(y)
\end{equation*}

\noindent this implies that $f=g$ and thus uniqueness of solutions.
\noindent For the second part of the proposition, we want to show that the full branch map obtained represents a circle map if and only if \eqref{The condition to extend a full branch to an expanding map} holds. This is because for a full branch map to lift to a circle map we need that the derivatives at the end points to coincide, as well as the left and right derivatives at the point $a$ and so by equation $(9)$ we need \eqref{The condition to extend a full branch to an expanding map} to hold.

\subsection{Modulus of continuity}
It just remains to show that \( f \) has \( \omega\) as an optimal modulus of continuity. 
Take $a\in(0,1)$ and consider a $C^1$ expanding diffeomorphism $f_{1}:[0,a]\to [0,1]$ admitting $\omega\in K$ as an optimal modulus of continuity and satisfying condition of \ref{The condition to extend a full branch to an expanding map}. By the previous section,  this extends to a Lebesgue preserving circle expanding map $f$, the regularity of the derivative on the first branch is by choice $\omega$-continuous, for the second branch $f_{2}$ we know that:

\begin{equation*}
    f_2'(x)=\frac{f_1'\Big(f_1^{-1}\big(f_2(x)\big)\Big)}{f_1'\Big(f_1^{-1}\big(f_2(x)\big)\Big)-1}.
\end{equation*}

\noindent Consider the map $\varphi=f_{1}^{-1}\circ f_{2}$. For $x,y\in[\frac{1}{2},1]$ we have:

\begin{equation*}
   |f_{2}'(x)-f_{2}'(y)| =|\frac{f_1'(\varphi(x))}{f_1'(\varphi(x))-1}-\frac{f_1'(\varphi(y))}{f_1'(\varphi(y))-1}|=|\frac{f_1'(\varphi(x))-f_1'(\varphi(y))}{(f_1'(\varphi(x))-1)(f_1'(\varphi(y))-1)}|.
\end{equation*}

\noindent Since $(f_1'(\varphi(x))-1)(f_1'(\varphi(y))-1)$ is bounded below away from $0$ because $f'>1$ and since $\varphi$ is Lipschitz (since it is $C^1$ on a compact interval)  we obtain:

\begin{equation*}
    \sup\limits_{|x-y|\leq t}|f_{2}'(x)-f_{2}'(y)|\simeq  \sup\limits_{|x-y|\leq t}|f_1'(\varphi(x))-f_1'(\varphi(y))|\simeq \omega(t).
\end{equation*}

\noindent We conclude that $f\in\Gamma^{1}_{\omega}(\mathbb{S}^1)$. Notice that in our construction, the choices we made to construct an example allow to construct uncountably many such element. This finishes the proof of the theorem.

\section{Proof of Theorem \ref{thm:dinidist1}}
 
\begin{proof}
    
\noindent Let $\omega\in K$ and $f\in \mathcal{F}_{\omega}$. For $k\in\mathbb{N}$ and by the chain rule we have that:

\begin{equation*}
|\log\dfrac{(f^k)'(x)}{(f^k)'(y)}|=|\sum\limits_{0\leq i\leq k-1} (\log (f'(f^i(x))-\log f'(f^i(y))|
\end{equation*}

\noindent Using mean value theorem, for every $0\leq i\leq k-1$ there exists 
\[
\lambda=\min\limits_{x\in S^1}\leq z_i\leq \sigma=\max\limits_{x\in S^1} |f'(x)|>1
\]
 such that
\begin{equation*}
    \sum\limits_{0\leq i\leq k-1} \log f'(f^i(x))-\log f'(f^i(y))= \sum\limits_{0\leq i\leq k-1}\dfrac{1}{z_{i}}(f'(f^i(x))-f'(f^i(y))).
\end{equation*}
 Now for every $k\in\mathbb{N}$, let us take the first partition element of order $k$, i.e.  $\omega^{k}_{1}=[0,r_{k}]$ where $\sigma^{-k}\leq r_k\leq\lambda^{-k}$. Let us take $y=0$ and $x_{k}\in\omega^{k}_{1}$ such that $f^{k}(x_k)\leq t_{\omega}$, for instance we can take $x_{k}=f^{-k}_{1}(t_{\omega})$ for \( k\)  large enough, where $f_1^{k}$ denotes the first branch of the $k$-\emph{th} iterate of $f$. From this we obtain
\begin{equation*}
    \left|\log\dfrac{(f^k)'(x_k)}{(f^k)'(0)}\right|\geq\dfrac{1}{\sigma}\sum\limits_{0\leq i\leq k-1} (f'(f^i(x_{k}))-f'(0))
\end{equation*}
and  we have that $f'(f^i(x_{k}))=2\rho(f^i(x_{k}))=2+2\omega_{f'}(f^i(x_k))$ and $f'(0)=2$ and so we obtain
\begin{equation*}
     \left|\log\dfrac{(f^k)'(x_k)}{(f^k)'(0)}\right|\geq\dfrac{2}{\sigma}\sum\limits_{0\leq i\leq k-1} \omega_{f'}(f^i(x_k)).
\end{equation*}
We have that $f^{i}(x_{k})=f^{i-k}(t_{\omega})$ and so we obtain 
\begin{equation*}
    f^{i}(x_{k})\geq C \sigma^{i-k}
\end{equation*}
and hence we get 
\begin{equation}\label{eq:lowerbound}
    \left|\log\dfrac{(f^k)'(x_k)}{(f^k)'(0)}\right|\geq\dfrac{2}{\sigma} \sum\limits_{0\leq i\leq k-1} \omega_{f'}(C\sigma^{i-k}). 
\end{equation}
We can now apply a Lemma from \cite{9}.

\begin{lemma}[\cite{9}]\label{lem:nonint}
$\omega\in K$ is not Dini-integrable if and only if  for every $\sigma>1$ we have
    \begin{equation*}
       \lim\limits_{k\to\infty} \sum_{1\leq i\leq k}\omega(\sigma^{-i})=\infty.
    \end{equation*}
\end{lemma}
Applying Lemma \ref{lem:nonint} to the inequality in \eqref{eq:lowerbound} we get that if if $\omega$ is not Dini-integrable we get 
\[
|\log\dfrac{(f^k)'(x_k)}{(f^k)'(0)}|\to\infty
\]
 and so $f$ has unbounded distortion. Conversely, if \( \omega \) is Dini-integrable then \( f\) has bounded distortion by  \cite{1} and so this finishes the proof.
\end{proof}

\section{Proof of Theorem \ref{thm:generic}}

To prove Theorem \ref{thm:generic} we first prove that the set of moduli which are not Dini-integrable are generic. First of all,  for every  $r\in\mathbb{N}$ let  $K_{r}\subset K$ be the space of moduli of continuity satisfying
    \begin{equation*}
         \int_{0}^{1}\dfrac{\omega(t)}{t}dt \leq r
    \end{equation*}
and let 
\[
K_{\infty}\coloneqq \bigcup_{r\in \mathbb N} K_{r}
\quad \text{ and } \quad 
K_{*}\coloneqq K\setminus K_{\infty}
\]
be the set of Dini integrable and non-Dini-integrable moduli of continuity respectively.

\begin{proposition}\label{prop:generic}
\( K_{*}\) is a  residual (dense $G_{\delta}$) set in the $C^0$-topology.
\end{proposition}

Before proving the proposition we prove two lemmas. 
\begin{lemma}
    The spaces $K_{r}$ are closed  subspaces of $K$ in the \( C^{0}\) topology.
\end{lemma}

    \begin{proof}
        Let $(\omega_{n})_{n\in\mathbb{N}}$ be a sequence in $K_{r}$ which converges uniformly to a map $\omega\in K$. For every $\epsilon>0$ the sequence ${\omega_{n}(t)}/{t}$ converges uniformly to ${\omega(t)}/{t}$ on $[\epsilon,1]$ and therefore
        \begin{equation}\label{eq:generic1}
            \int_{\epsilon}^{1}\dfrac{\omega(t)}{t}dt=\lim\int_{\epsilon}^{1}\dfrac{\omega_n(t)}{t}dt\leq r.
        \end{equation}
Since \eqref{eq:generic1} holds for every $\epsilon>0$ we deduce that $\omega\in K_{r}$ and hence $K_{r}$ is closed.
    \end{proof}
   
\begin{lemma}
    The spaces $K_{r}$ have empty interior in \( K \). 
\end{lemma}


\begin{proof} We will show that \( K_{*}\) is dense in \( K \), which clearly implies the statement. 
Let $\omega\in K_{\infty}$  and  $\omega_0\in K_{*}$ such that $|\omega_{0}|_{\infty}=1$.  For every $\epsilon$ we have that $\omega_{\epsilon}=\omega+\epsilon\omega_{0}\in K_{*}$  and $|\omega_{\epsilon}-\omega|_{\infty}=\epsilon$. This implies that \( K_{*}\) intersects every open set in $K$ and hence is dense.
    \end{proof}

\begin{proof}[Proof of Proposition \ref{prop:generic}]
Since $K$ is complete in the uniform topology and $K_{\infty}$ is a countable union of closed sets with empty interior we conclude by Baire's category theorem that \( K_{*}\)  is a dense $G_{\delta}$ set. 
\end{proof}

\begin{proof}[Proof of Theorem \ref{thm:generic}] 

By Theorem \ref{thm:dinidist1} for every \(\omega\in K \) there exists an uncountable family $\mathcal{F}_{\omega}$ whose elements have unbounded distortion if and only if $\omega\in K_{*}$. In particular, for every \(\omega\in K \) we have that
\begin{equation}\label{eq:intersection}
\mathcal{F}_{\omega}\cap  E^1_{\omega}(\mathbb{S}^1) \neq \emptyset.
\end{equation}
By the axiom of choice we can select a set $\Gamma\subset E^1 (\mathbb{S}^1)$ consisting of one element of the intersection 
\eqref{eq:intersection} for each \( \omega\in K \). We will show that \( \Gamma\) satisfies the conclusions in the statement of the Theorem.

 Let $\widetilde{K}$ denote the quotient space $K/\simeq$ equipped with the quotient topology and $\pi:K\to\widetilde{K}$ the quotient map. Notice that the equivalence classes preserve Dini-integrability and therefore we have well defined subsets \( \widetilde K_{
 \infty}\) and \( \widetilde K_{*}\) such that \( \pi(K_{\infty})= \widetilde K_{\infty} \) 
and \( \pi(K_{*})= \widetilde K_{*} \). By Proposition \ref{prop:generic} the set \( K_{*}\) is residual in \( K\) in the \( C^{0}\) topology and since the projection \( \pi \) is an open mapping it follows that  \( \widetilde K_{*}\) is residual in \( \widetilde K\) in the quotient topology. 

Finally, let \( \mathcal M: \Gamma \to K \) denote the map which assigns to each \( f\in \Gamma\) the modulus \( \omega_{f'}\in K \). By definition \( \mathcal M  \) is injective and continuous in the \( C^{1+mod}\) topology on \( \Gamma\) and therefore the composition $\pi\circ\mathcal{M}:\Gamma\to \Tilde{K}$ is a homeomorphism. It follows that $(\pi\circ\mathcal{M})^{-1}(\widetilde K_{*})$ is $G_{\delta}$ is residual in $\Gamma$.  
\end{proof}

\printbibliography

\end{document}